\documentclass{amsart}
\usepackage[utf8]{inputenc}
\usepackage{amsmath, amsthm}
\usepackage{amssymb}
\usepackage{todonotes}
\usepackage{tikz-cd}
\newtheorem{theorem}{Theorem}
\newtheorem*{theorem*}{Theorem}
\newtheorem*{definition}{Definition}
\newtheorem{lemma}[theorem]{Lemma}
\newtheorem{prop}[theorem]{Proposition}

\newtheorem{corollary}[theorem]{Corollary}

\newcommand{\col}{\mathrm{Col}}

\newcommand{\la}{\langle}
\newcommand{\ra}{\rangle}
\newcommand{\rest}{{\restriction}}

\newcommand{\rank}{\mathrm{rank}}

\newcommand{\bbB}{\mathbb{B}}

\newcommand{\bbE}{\mathbb{E}}

\newcommand{\bbL}{\mathbb{L}}

\newcommand{\bbP}{\mathbb{P}}
\newcommand{\bbQ}{\mathbb{Q}}

\newcommand{\p}{\mathcal{P}}

\DeclareMathOperator{\dom}{dom}
\DeclareMathOperator{\ran}{ran}

\title{Saturated ideals from Laver collapses}
\author{Monroe Eskew}
\address{Kurt G\"odel Research Center,
University of Vienna,
Vienna, Austria.}
\email{monroe.eskew@univie.ac.at}
\thanks{This work was supported by the Austrian Science Fund (FWF) through Project P34603.}

\begin{document}

\begin{abstract}
    Addressing a question of Shioya, we show that two-step iterations of the Laver collapse can force saturated ideals and Chang conjectures.
\end{abstract}

\maketitle

\section{Introduction}

In \cite{shioyaeaston}, Shioya introduced the Easton collapse and showed that a two-step iteration of these collapses, $\bbE(\mu,\kappa) * \dot \bbE(\kappa,\lambda)$, can introduce a saturated ideal on $\kappa$, starting from a model in which $\kappa$ is huge.  The author and Hayut \cite{eh} showed that the same poset forces Chang's conjecture $(\mu^{++},\mu^+)\twoheadrightarrow(\mu^+,\mu)$.  This is arguably the simplest known method for forcing these properties of successor cardinals, and it lends itself quite easily to iteration for obtaining many instances simultaneously \cite{eh,shioya2020}.  Shioya \cite{shioyaeaston} also raised the question of whether similar effects may be achieved by a two-step iteration of the Laver collapse $\bbL(\mu,\kappa) * \dot \bbL(\kappa,\lambda)$, which was introduced in \cite{laver}.  Here, we answer this affirmatively by constructing a projection between Laver collapses indexed by different cardinals.  Our main result is the following, which is proved in a more specific form as Theorem~\ref{main}.

\begin{theorem*}
If GCH holds, $j : V \to M$ is a huge embedding with $\mathrm{crit}(j) = \kappa$ and $j(\kappa) = \lambda$, and $\mu<\kappa$ is regular, then $\bbL(\mu,\kappa) * \dot\bbL(\kappa,\lambda)$ forces $(\mu^{++},\mu^+) \twoheadrightarrow (\mu^+,\mu)$ and that there is a saturated ideal on $\kappa$.
\end{theorem*}

The assumption of GCH can be eliminated, but we leave those details to the reader.
The projection we construct has a very simple quotient.  Lemma~\ref{laverproj} shows that the Laver collapse has the remarkable property that for regular cardinals $\mu<\kappa<\lambda$ with $\lambda$ inaccessible, $\bbL(\mu,\lambda)$ is forcing-equivalent to $\bbL(\kappa,\lambda) \times \col(\mu,\kappa)$.

Our notation and terminology should be standard.  We assume familiarity with the basics of large cardinals and forcing.  Let us recall some important notions that we will use.  For a regular cardinal $\kappa$, a $\kappa$-complete ideal $I$ on $\kappa$ is called \emph{saturated} if the quotient algebra $\p(\kappa)/I$ has the $\kappa^+$-chain condition $(\kappa^+$-c.c.).  For basic properties of saturated ideals, see \cite{kanamori}.  For partially ordered sets $\bbP,\bbQ$ with maximal elements $1_\bbP,1_\bbQ$, a map $\pi : \bbQ \to \bbP$ is called a \emph{projection} if it is order-preserving, sends $1_\bbQ$ to $1_\bbP$, and has the property that whenever $p \leq \pi(q)$, there is some $q' \leq q$ such that $\pi(q') \leq p$.  It is not hard to check that if $H \subseteq \bbQ$ is $\bbQ$-generic over $V$, then $\pi[H]$ generates a filter that is $\bbP$-generic over $V$, which we denote by $\pi(H)$.  It is a folklore fact that whenever $\pi : \bbQ \to \bbP$ is a projection, forcing with $\bbQ$ is equivalent to first taking a generic filter $G \subseteq\bbP$ and then forcing with the quotient $\bbQ/G = \pi^{-1}[G]$.  Another way a poset $\bbQ$ may be decomposed into a two-step iteration is when there is a \emph{complete embedding} $e : \bbP \to \bbQ$, which is an order-preserving map that preserves maximal antichains (see \cite{kanamori}).  A complete embedding is called a \emph{dense embedding} if its range is dense in the target poset, and in this case, the two posets yield the same generic extensions.

\section{The projection}

Recall that a set of ordinals $X$ is \emph{Easton} when for all regular $\alpha$, $\sup(X \cap \alpha)<\alpha$.

\begin{definition}[Laver collapse]
Suppose $\kappa$ is a regular cardinal and $\lambda>\kappa$.  The Laver collapse $\bbL(\kappa,\lambda)$ is the collection of partial functions $p$ with the following properties:
\begin{enumerate}
    \item $\dom(p)$ is an Easton set of regular cardinals $\delta$ such that $\kappa<\delta<\lambda$.
    \item There is $\xi < \kappa$ such that for all $\delta \in \dom(p)$, $p(\delta)$ is a function with domain contained in $\xi$ and range contained in $\delta$.
    \end{enumerate}
We put $p \leq q$ when $\dom(q) \subseteq \dom(p)$ and for all $\delta \in \dom(q)$, $q(\delta) \subseteq p(\delta)$.
\end{definition}

\begin{lemma}
If $\kappa$ is regular and $\lambda>\kappa$, then $\bbL(\kappa,\lambda)$ is $\kappa$-directed-closed 
and collapses all cardinals between $\kappa$ and $\lambda$.
\end{lemma}

\begin{lemma}
If $\kappa$ is regular and $\lambda>\kappa$ is Mahlo, then $\bbL(\kappa,\lambda)$ is $\lambda$-c.c.
\end{lemma}
\begin{proof}
    Let $S$ be the set of regular cardinals below $\lambda$.  Since $\lambda$ is Mahlo, $S$ is stationary in $\lambda$.  Let $A \subseteq \bbL(\kappa,\lambda)$ be a maximal antichain.  There is a club $C \subseteq \lambda$ such that for all $\alpha \in C$, $A \cap \bigcup_{\beta<\alpha} \bbL(\kappa,\beta)$ is a maximal antichain in $\bigcup_{\beta<\alpha} \bbL(\kappa,\beta)$.  If $\alpha \in S$, then $\bigcup_{\beta<\alpha} \bbL(\kappa,\beta) = \bbL(\kappa,\alpha)$ since all $p \in \bbL(\kappa,\alpha)$ have domain smaller than $\alpha$.  We claim that $A \subseteq \bbL(\kappa,\alpha)$, where $\alpha = \min(C \cap S)$, which implies $|A|<\lambda$.  Otherwise, there is some $p \in A \setminus \bbL(\kappa,\alpha)$.  But $p \rest \alpha \in \bbL(\kappa,\alpha)$, and thus $p$ is compatible with some $q \in A \cap \bbL(\kappa,\alpha)$, a contradiction.
\end{proof}

For cardinals $\kappa,\lambda$, the poset $\col(\kappa,\lambda)$ is defined as the collection of functions from $x \subseteq \kappa$ into $\lambda$, where $|x| <\kappa$, ordered by $p \leq q \leftrightarrow p \supseteq q$.  When $\kappa$ is regular, the set $\{ p \in \col(\kappa,\lambda)$ : $\dom(p)$ is an ordinal$\}$ is dense.

\begin{lemma}
\label{laverproj}
    Suppose $\mu<\kappa$ are regular, $\lambda>\kappa$, and $\delta^{<\kappa} = \delta$ for all regular $\delta\in(\kappa,\lambda)$.  Then there is a dense embedding from a dense subset of $\bbL(\mu,\lambda) \restriction [\kappa,\lambda)$ to $\col(\mu,\kappa) \times \bbL(\kappa,\lambda)$.
\end{lemma}

\begin{proof}
    The first coordinate of $\bbL(\mu,\lambda) \restriction [\kappa,\lambda)$ is $\col(\mu,\kappa)$.  From any condition $p \in \col(\mu,\kappa)$ whose domain is an ordinal, we derive a closed bounded subset $C_p \subseteq \kappa$ of ordertype $\dom(p)+1$ recursively as follows.  Let $C_p(0) = 0$.  Suppose $\alpha \leq \dom(p)$ and we have defined the first $1+\alpha$ elements of $C_p$.  If $\alpha$ is a limit, let $C_p(\alpha) = \sup_{\beta<\alpha} C_p(\beta)$.  If $\alpha = \beta+1$, let $C_p(\alpha) = C_p(\beta) + 1 + p(\beta)$.  Note that $p$ can be recovered from $C_p$, and if $p \leq q$, then $C_q$ is an initial segment of $C_p$.

     For each regular $\delta\in(\kappa,\lambda)$ and each pair $\alpha<\beta<\kappa$ choose a bijective enumeration $\la f^{\alpha,\beta,\delta}_i : i < \delta \ra$ of all functions from $[\alpha,\beta)$ to $\delta$.  Here, we use the assumption that $\delta^{<\kappa} = \delta$.

     Let us define a map $\pi : \bbL(\mu,\lambda) \restriction [\kappa,\lambda) \to \col(\mu,\kappa) \times \bbL(\kappa,\lambda)$ on the dense set of conditions $r$ such that $\kappa \in \dom(r)$ and there is $\gamma < \mu$ such that $\dom(r(\delta)) = \gamma$ for all $\delta \in \dom(r)$.   
     For such an $r$, let $p = r(\kappa)$ and let $q = r \rest (\kappa,\lambda)$.  Let $X = \dom(q)$, and let $\gamma = \dom(p)$, which is also equal to $\dom(q(\delta))$ for all $\delta \in X$.  Let $\pi(r) = \la p,s \ra$, where $s$ is defined as follows:  $\dom(s) = X$, and for each $\delta \in X$, $s(\delta) : \sup(C_p) \to \delta$ is the function such that for all $i<\gamma$,
     $$s(\delta) \restriction [C_p(i),C_p(i+1)) = f^{C_p(i),C_p(i+1),\delta}_{q(\delta)(i)}$$

     It is clear that $\pi$ is an order-preserving injection. 
     To show that $\pi$ is an order-isomorphism with its range, suppose $\pi(r_1) \leq \pi(r_0)$.  For $i=0,1$, let $q_i = r_i \rest (\kappa,\lambda)$, $\la p_i,s_i \ra = \pi(r_i)$, and $\gamma_i = \dom(p_i)$.  We have that $p_1 \leq p_0$, and thus $C_{p_0}$ is an initial segment of $C := C_{p_1}$.  Since $s_1 \leq s_0$, it follows that for all $\delta \in \dom(s_0)$ and all $j<\gamma_0$, 
          $$f^{C(j),C(j+1),\delta}_{q_0(\delta)(j)} = f^{C(j),C(j+1),\delta}_{q_1(\delta)(j)}$$
    Since the enumerations $\la f^{\alpha,\beta,\delta}_i : i < \delta \ra$ are one-to-one, it must be the case that for all $\delta \in \dom(q_0) = \dom(s_0)$ and all $j<\gamma_0$, $q_0(\delta)(j) = q_1(\delta)(j)$.  Thus $q_1 \leq q_0$, and so $r_1 \leq r_0$.

     Finally, we show that $\ran(\pi)$ is dense.  Suppose $\la p,s \ra \in \col(\mu,\kappa) \times \bbL(\kappa,\lambda)$.  Let $\gamma = \dom(p)$ and let $X = \dom(s)$.  By strengthening the condition if necessary, we may assume that for all $\delta \in X$, $\dom(s(\delta))=\sup(C_p)$.  We can find $q \in \bbL(\mu,\lambda) \rest (\kappa,\lambda)$ such that $\dom(q) = X$, $\dom(q(\delta)) = \gamma$ for all $\delta \in X$, and for all $\delta \in X$ and $i<\gamma$, $q(\delta)(i)$ is the unique ordinal $j<\delta$ such that
     $$s(\delta) \restriction [C_p(i),C_p(i+1)) = f^{C_p(i),C_p(i+1),\delta}_j$$
     If $r \in \bbL(\mu,\lambda) \rest [\kappa,\lambda)$ is the condition such that $r(\kappa) = p$ and $r \restriction (\kappa,\lambda) = q$, then $\pi(r) = \la p,s \ra$.
     \end{proof}

\begin{corollary}
\label{densemb}
    Suppose $\mu<\kappa$ are regular, $\lambda>\kappa$, and $\delta^{<\kappa} = \delta$ for all regular $\delta\in(\kappa,\lambda)$.  Then there is a dense  $D \subseteq \bbL(\mu,\lambda)$ and a projection $\pi_\kappa : D \to \bbL(\kappa,\lambda)$ such that 
    $$ p \mapsto \la p \rest \kappa, p(\kappa), \pi_\kappa(p) \ra $$
    is a dense embedding from $D$ to $\bbL(\mu,\kappa) \times \col(\mu,\kappa) \times \bbL(\kappa,\lambda)$.
    \end{corollary}

\section{The main result}

Suppose $\bbP$ is a poset and $\dot\bbQ$ is a $\bbP$-name for a poset.  The \emph{term forcing} $T(\bbP,\dot\bbQ)$ is the collection of all $\bbP$-names for elements of $\dot\bbQ$ (which are members of $H_\theta$, where $\theta = (2^{\rank(\{\bbP,\dot\bbQ\})})^+$), ordered by $\dot q_1 \leq \dot q_0$ when $1_\bbP \Vdash \dot q_1 \leq_\bbP \dot q_0$.  This notion as well as the following fact are due to Laver.

\begin{prop}[Laver]
\label{laver}
The identity map is a projection from $\bbP \times T(\bbP,\dot\bbQ)$ to $\bbP * \dot\bbQ$.
\end{prop} 
Proof of the following is exactly as in \cite{shioya2020}.

\begin{lemma}
\label{term}
    Suppose $\kappa$ is regular, $\mathbb P$ is a $\kappa$-c.c.\ poset of size $\kappa$, $\lambda>\kappa$, and $\delta^{<\kappa} = \delta$ for all regular $\delta \in (\kappa,\lambda)$.  Then there is a dense embedding from $\bbL(\kappa,\lambda)$ to $T(\mathbb P,\dot{\bbL}(\kappa,\lambda))$.
\end{lemma}

\begin{proof}
    For each regular $\delta \in (\kappa,\lambda)$, choose an injective enumeration of a complete set of representatives of $\bbP$-names for ordinals below $\delta$, $\la \tau_i^\delta : i < \delta \ra$.  For each $q \in \bbL(\kappa,\lambda)$, let $\dot q$ be a $\bbP$-name for the function such that $1 \Vdash \dom \dot q = \dom \check q$, for all $\delta \in \dom (q)$, $1 \Vdash \dom \dot q(\delta) = \dom \check q(\delta)$, and for all $\delta \in \dom(q)$ and $\alpha \in \dom q(\delta)$, $1 \Vdash \dot q(\delta)(\alpha) = \tau_{q(\delta)(\alpha)}^\delta$.

    It is clear that the map $q \mapsto \dot q$ is an order-isomorphism with its range.  
    To show it is dense, suppose $\sigma \in T(\bbP,\dot\bbL(\kappa,\lambda))$.  Using the $\kappa$-c.c., there is an Easton set $X \subseteq \lambda$ and an ordinal $\gamma < \kappa$ such that $1 \Vdash \dom \sigma \subseteq \check X$, and for each $\delta \in \dom \sigma$, $\dom \sigma(\delta) \subseteq \check\gamma$. For $\delta \in X$ and $\alpha<\gamma$, let $\sigma_{\delta\alpha}$ be a $\bbP$-name for an ordinal that is forced to be equal to $\sigma(\delta)(\alpha)$ if this is defined, and otherwise equal to 0.  
    Define $q \in \bbL(\kappa,\lambda)$ as follows:  $\dom q = X$, $\dom q(\delta) = \gamma$ for all $\gamma \in X$, and for $\la \delta,\alpha \ra \in X \times \gamma$, $q(\delta)(\alpha)$ is the unique $i$ such that $1 \Vdash \sigma_{\delta\alpha} = \tau_i^\delta$.  It is straightforward to check that $1 \Vdash \dot q \leq \sigma$.
\end{proof}

Recall that Chang's conjecture $(\kappa_1,\kappa_0) \twoheadrightarrow (\mu_1,\mu_0)$ states that for all structures $\frak A$ on $\kappa_1$ in a countable language, there is a elementary substructure $\frak B \prec \frak A$ of size $\mu_1$ such that $| \frak B \cap \kappa_0 | = \mu_0$.  The following is the main result of this article:  

\begin{theorem}
\label{main}
Assume GCH and $j : V \to M$ is a huge embedding with $\mathrm{crit}(j) = \kappa$ and $j(\kappa) = \lambda$.  Let $\mu<\kappa$ be regular.  Then $\bbL(\mu,\kappa) * \dot\bbL(\kappa,\lambda)$ forces:
\begin{enumerate}
    \item $\kappa = \mu^+$ and $\lambda = \mu^{++}$.
    \item $(\mu^{++},\mu^+) \twoheadrightarrow (\mu^+,\mu)$.
    \item There is a normal ideal $I$ on $\kappa$ such that $\p(\kappa)/I \cong   \bbQ \times \col(\mu,\kappa)$, where $\bbQ$ is a $\mu$-closed, $\kappa$-distributive, $\lambda$-c.c.\ poset.
\end{enumerate}
\end{theorem}

\begin{proof}
    Since $\kappa$ and $\lambda$ are both Mahlo, $\bbL(\mu,\kappa)$ is $\kappa$-c.c., and $\bbL(\nu,\lambda)$ is $\lambda$-c.c.\ for $\nu=\mu,\kappa$.  This implies that $\kappa = \mu^+$ and $\lambda = \mu^{++}$ after forcing with $\bbL(\mu,\kappa) * \dot\bbL(\kappa,\lambda)$.
    
    By Lemma~\ref{term}, there is a dense embedding $\pi_T : \bbL(\kappa,\lambda) \to T(\bbL(\mu,\kappa),\dot\bbL(\kappa,\lambda))$, and by Proposition~\ref{laver}, $\mathrm{id} \times \pi_T : \bbL(\mu,\kappa) \times \bbL(\kappa,\lambda) \to \bbL(\mu,\kappa) * \dot\bbL(\kappa,\lambda)$ is a projection.
    Thus we have the following chain of dense embeddings (denoted by $\cong$) and projections (denoted by $\rightarrow$):
\begin{align*}
    \bbL(\mu,\lambda)   &\cong \bbL(\mu,\kappa) \times \col(\mu,\kappa) \times \bbL(\kappa,\lambda) &\text{(by Corollary~\ref{densemb})}\\     
                        &\cong \bbL(\mu,\kappa) \times \bbL(\kappa,\lambda) \times \col(\mu,\kappa) & \\
                        &\cong \bbL(\mu,\kappa) \times T(\bbL(\mu,\kappa),\dot\bbL(\kappa,\lambda)) \times \col(\mu,\kappa) & \text{(by Lemma~\ref{term})} \\
                        &\rightarrow \bbL(\mu,\kappa) * \dot\bbL(\kappa,\lambda)                        &\text{(by Proposition~\ref{laver})}
\end{align*}
    Let $\pi : \bbL(\mu,\lambda) \to \bbL(\mu,\kappa) * \dot\bbL(\kappa,\lambda)$ be the projection given by the composition of these maps, i.e.
    $$\pi : p \mapsto \la p \rest \kappa, p(\kappa), \pi_\kappa(p) \ra \mapsto \la p \rest \kappa,\pi_T\circ\pi_\kappa(p) \ra,$$
    where $\pi_\kappa : \bbL(\mu,\lambda) \to \bbL(\kappa,\lambda)$ is as in Corollary~\ref{densemb}.
    If $G * H \subseteq \bbL(\mu,\kappa)*\dot\bbL(\kappa,\lambda)$ is generic over $V$, then the quotient forcing $\pi^{-1}[G*H]$ is easily seen to be equivalent to $\bbQ \times \col(\mu,\kappa)$, where 
    $$\bbQ = \{ q \in \bbL(\kappa,\lambda)^V : \pi_T(q)^G \in H \}.$$
    $\bbQ$ is $\mu$-closed since $\bbL(\mu,\kappa)$ adds no ${<}\mu$-sequences, and if $\delta<\mu$ and $\la q_i : i < \delta \ra$ is a descending sequence in $\bbL(\kappa,\lambda)^V$ such that $\{\pi_T(q_i)^G : i <\delta \} \subseteq H$, then $\pi_T(\inf_i q_i)^G \in H$ as well.
    $\bbQ$ is $\kappa$-distributive since by Easton's Lemma, $\bbL(\kappa,\lambda)^V$ is $\kappa$-distributive in $V[G]$, and $\bbQ$ is equivalent to the quotient forcing of the projection $\mathrm{id} \times \pi_T : \bbL(\mu,\kappa) \times \bbL(\kappa,\lambda) \to \bbL(\mu,\kappa) * \dot\bbL(\kappa,\lambda)$.
    
    To show that $(\mu^{++},\mu^+) \twoheadrightarrow (\mu^+,\mu)$ holds in $V[G*H]$, we lift the embedding.
    Forcing with the quotient $\pi^{-1}[G*H]$, let $G' \subseteq \bbL(\mu,\lambda)$ be generic over $V$ such that $\pi(G')=G*H$.  
    Every $q \in H$ is a function defined on an Easton set $X_q \subseteq (\kappa,\lambda)$ and such that there exists $\gamma_q < \kappa$ with $\dom q(\delta) \subseteq \gamma_q$ for all $\delta \in X_q$.    
    Let $X = \bigcup \{ j(X_q) : q \in H \}$
    and let $q^* : X \to M[G']$ be defined by $q^*(\delta) = \bigcup \{ j(q)(\delta) : q \in H \}$, which is a function from $\kappa$ to $\delta$.  Thus $q^* \in \bbL(\lambda,j(\lambda))^{M[G']}$, and $q^* = \inf j[H]$.  If $H' \subseteq \bbL(\lambda,j(\lambda))^{M[G']}$ is generic over $V[G']$ with $q^* \in H'$, then $j$ can be lifted to $j : V[G * H] \to M[G' * H']$.
    
    Now let $\frak A \in V[G*H]$ be any structure on $\lambda$ in a countable language.  Then $j[\lambda] \prec j(\frak A)$, and in $M[G'*H']$, $|j[\lambda]| = \lambda = j(\kappa)$, and $| j[\lambda] \cap j(\kappa) | = |\kappa| = \mu$.  By elementarity, it holds in $V[G*H]$ that there is $\frak B \prec \frak A$ such that $|\frak B| = \kappa$ and $|\frak B \cap \kappa| = \mu$.

    To show the final claim, we lift a derived almost-huge embedding.  Using \cite[Theorem 24.11]{kanamori}, we can derive such an embedding $i : V \to N$ from $j$, with $\mathrm{crit}(i) = \kappa$, $i(\kappa) = \lambda$, $N^{<\lambda} \subseteq N$, and $\lambda < i(\lambda) < \lambda^+$.  Thus whenever $G' \subseteq \bbL(\mu,\lambda)$ is generic and $G*H = \pi(G')$, $i$ lifts to $i : V[G] \to N[G']$ with $i(G) = G'$, and $N[G']$ is ${<}\lambda$-closed in $V[G']$.  For $\alpha<\lambda$, let $q^*_\alpha = \inf j[H \rest \alpha]$.  Using $|i(\lambda)| = \lambda$ in $V[G']$, enumerate the dense open subsets of $\bbL(\lambda,i(\lambda))$ in $N[G']$ as $\la D_\alpha : \alpha<\lambda \ra$.  Then construct a descending sequence $\la r_\alpha : \alpha < \lambda \ra$ such that for all $\alpha<\lambda$, $r_{\alpha+1} \leq q^*_\alpha$ and $r_{\alpha+1} \in D_\alpha$.  Let $H'$ be the $M[G']$-generic filter generated by $\la r_\alpha : \alpha < \lambda \ra$, and note that $H' \in V[G']$.  Then $i$ lifts further to $i : V[G*H] \to N[G'*H']$, with $i(G*H) = G'*H'$.

    In $V[G*H]$, let $\bbB$ be the Boolean completion of $\bbQ \times \col(\mu,\kappa)$ and let $I = \{ A \subseteq \kappa : 1_\bbB \Vdash \kappa \notin i(A) \}$, which is a $\kappa$-complete normal ideal.  Define a map $e : \p(\kappa)/I \to \bbB$ by $e([A]_I) = || \kappa \in i(A)||$.  It is easy to check that $e$ preserves order and incompatibility, implying that  $\p(\kappa)/I$ is $\lambda$-c.c.  It then follows that $e$ is a complete embedding, since for any maximal antichain $\la [A_\alpha]_I : \alpha <\kappa \ra \subseteq \p(\kappa)/I$, the diagonal union $\nabla_\alpha A_\alpha$ is in the class $[\kappa]_I$, and so it is forced that $\kappa \in i(A_\alpha)$ for some $\alpha<\kappa$, implying that $\{ e([A_\alpha]_I) : \alpha<\kappa \}$ is maximal in $\bbB$.

    In order to show that $\p(\kappa)/I$ is equivalent to $\bbB$, it suffices to show that whenever $G' \subseteq \bbL(\mu,\kappa)$ is generic, $G*H = \pi(G')$, $K$ is the induced generic for $\bbB$, and $U = e^{-1}[K]$, then $K$ can be defined in $V[G][H][U]$.  Note that $U = \{ A \in \p(\kappa)^{V[G*H]} : \kappa \in i(A) \}$.  Let $N_U = \mathrm{Ult}(V[G*H],U)$ and let $i_U : V[G*H] \to N_U$ be the standard embedding.  If $k : N_U \to N[G'*H']$ is defined by $k([f]_U) = i(f)(\kappa)$, then $k$ is elementary and $i = k \circ i_U$.
    If $k = \mathrm{id}$, then $i_U(G) = i(G) = G'$.
    If $k \not= \mathrm{id}$, then $\mathrm{crit}(k)$ is a cardinal of $N_U$ above $\mu$, and since $(\mu^+)^{N_U}=i_U(\kappa) = i(\kappa) = \lambda$, $\mathrm{crit}(k) > \lambda$.  Thus also in this case, $i_U(G) = k(i_U(G)) = i(G) = G'$.  Therefore, in either case, $K$ is definable in $V[G][H][U]$ from the parameters $G,H,U,\pi,e$.
    \end{proof}

    \bibliographystyle{amsplain.bst}
    \bibliography{laver.bib}
\end{document}